\documentclass[a4paper,abstracton]{scrartcl}

\usepackage{amsfonts,amsthm,amssymb,amsmath}
\usepackage[latin1]{inputenc}
\usepackage[T1]{fontenc}
\usepackage{graphicx}
\usepackage[inline]{enumitem}
\usepackage{ifpdf}
   \ifpdf
   \fi
   
\usepackage{enumitem}
\usepackage{hyperref}

\newtheorem{thm}{Theorem}[section]

\theoremstyle{remark}

\newtheorem{con}[thm]{Conjecture}

\newcommand{\Aut}{\mathrm{Aut \,}}

\title{Local finiteness, distinguishing numbers and Tucker's conjecture}
\author{Florian Lehner\thanks{The author acknowledges the support of the Austrian Science Fund (FWF), project W1230-N13.} \ and Rögnvaldur G.~Möller}

\begin{document}
\maketitle

\begin{abstract}
A distinguishing colouring of a graph is a colouring of the vertex set such that no non-trivial automorphism preserves the colouring. Tucker conjectured that if every non-trivial automorphism of a locally finite graph moves infinitely many vertices, then there is a distinguishing $2$-colouring.

We show that the requirement of local finiteness is necessary by giving a non-locally finite graph for which no finite number of colours suffices.
\end{abstract}

\section{Introduction}
\label{sec:intro}

A colouring of the vertices of a graph $G$ is called {\em distinguishing} if no  non-trivial automorphism of $G$ preserves the colouring. This notion was first studied by Albertson and Collins \cite{MR1394549}, motivated by a recreational mathematics problem posed Rubin \cite{rubin}. 

While a distinguishing colouring clearly exists for every graph (simply colour every vertex with a different colour), finding a distinguishing colouring with the minimum number of colours can be challenging.

For infinite graphs one of the most intriguing questions is whether or not the following conjecture of Tucker \cite{MR2776826} is true.  

\begin{con}
\label{con:tucker}
Let $G$ be an infinite, connected, locally finite graph with infinite motion. Then there is a distinguishing $2$-colouring of $G$.
\end{con}

This conjecture can be viewed as a generalisation of a result on finite graphs due to Russell and Sundaram \cite{MR1617449}.  It is known to be true for many classes of infinite graphs including trees \cite{MR2302536}, tree-like graphs \cite{MR2302543}, and graphs with countable automorphism group \cite{istw}. In \cite{smtuwa} it is shown that graphs satisfying the so-called distinct spheres condition have infinite motion as well as distinguishing number two. Examples for such graphs include leafless trees, graphs with infinite diameter and primitive automorphism group, vertex-transitive graphs of connectivity $1$, and Cartesian products of graphs where at least two factors have infinite diameter. It is also known that Conjecture~\ref{con:tucker} is true for graphs fulfilling certain growth conditions \cite{growth}. In \cite{lehner-randomcolouring} it is shown that for locally finite graphs random colourings have a good chance of being distinguishing. 

Many of the above results also hold for non-locally finite graphs which raises the question, whether the condition of local finiteness in Tucker's conjecture can be dropped. 

A first indication, that local finiteness may be necessary has been given in the setting of permutation groups acting on countable sets. Here, instead of considering the automorphism group of a graph acting on the vertex set, we consider (faithful) group actions. A generalization of Conjecture \ref{con:tucker} to this setting has been given by Imrich et al.~\cite{istw}.

\begin{con}
\label{con:groups}
Let $\Gamma$ be a closed, subdegree finite permutation group on a set $S$.  Then there is a distinguishing $2$-colouring of $S$.
\end{con}

For this generalization subdegree finiteness (which plays the role of local finiteness) is known to be necessary \cite{MR2587751}. 

In this short note we show that local finiteness is also necessary in the graph case. More precisely we give a non-locally finite, arc transitive graph with infinite motion which does not admit a distinguishing colouring with any finite number of colours.

\section{Preliminaries}

Throughout this paper we will use Greek letters for group related variables while the Latin alphabet will be reserved for sets on which the group acts. 

Let $S$ be a countable set and let $\Gamma$ be a group acting faithfully (i.e. the identity is the only group element which acts trivially) on $S$ from the left. The image of a point $s \in S$ under an element $\gamma \in \Gamma$ is denoted by $\gamma s$.

The \emph{stabilizer} of $s$ in $\Gamma$ is defined as the subgroup $\Gamma_s = \{\gamma \in \Gamma \mid \gamma s = s\}$.
We say that $\Gamma$ is \emph{subdegree finite} if for every $s \in S$ all orbits of $\Gamma_s$ are finite.


The \emph{motion} of an element $\gamma \in \Gamma$ is the number (possibly infinite) of elements of $S$ which are not fixed by $\gamma$. The \emph{motion of the group $\Gamma$} is the minimal motion of a non-trivial element of $\Gamma$. Notice that the motion is not necessarily finite, in fact all groups considered in this paper have infinite motion.  The \emph{motion of a graph $G$} is the motion of $\Aut G$ acting on the vertex set.

Let $C$ be a (usually finite) set. A \emph{$C$-colouring of $S$} is a map $c \colon S \to C$.  Given a colouring $c$ and $\gamma \in \Gamma$ we say that \emph{$\gamma$ preserves $c$} if $c(\gamma s) = c(s)$ for every $s \in S$.  Call a colouring \emph{distinguishing} if no non-trivial group element preserves the colouring.

\section{The example}

The construction that we use relies on the following result from  \cite{MR2587751} which also shows that there are permutation groups on a countable sets whose distinguishing number is infinite. The proof uses a standard back-and-forth argument, see for example \cite[Sections 9.1 and 9.2]{zbMATH01230795} and \cite[Sections 2.8 and 5.2]{zbMATH00044603}.

\begin{thm}[Laflamme et al.~\cite{MR2587751}]
\label{thm:dist-q}
Let $\Gamma$ be the group of  order automorphisms of $\mathbb Q$ (i.e.\ bijective, order preserving functions $\gamma\colon \mathbb Q \to \mathbb Q$). Then $\Gamma$ has infinite motion but no distinguishing colouring with finitely many colours. \hfill \qedsymbol
\end{thm}

Clearly the group $\Gamma$ of the above theorem is the full automorphism group of the following directed graph: take $\mathbb Q$ as vertex set and draw an edge from $q$ to $r$ if $q \leq r$. The underlying undirected graph is the complete countable graph which also has infinite distinguishing number but finite motion.




\begin{figure}
\begin{center}
\includegraphics{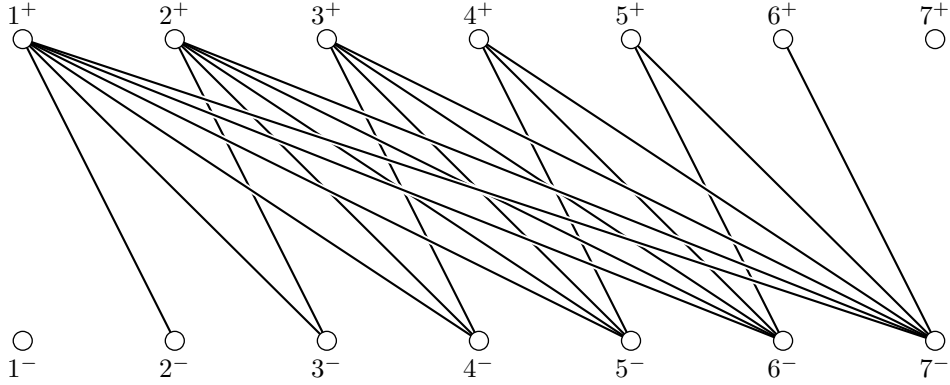}
\end{center}
\caption[An induced subgraph of the graph from Theorem \ref{thm:counterexample}.]{An induced subgraph of the graph in Theorem \ref{thm:counterexample}. Note that edges only go from top left to bottom right. By the definition of the graph all such edges are present and every edge is of this type.}
\label{fig:induced}
\end{figure}

\begin{thm}
\label{thm:counterexample}
There is a countable, connected, arc transitive graph with infinite motion which has no distinguishing colouring with a finite number of colours.
\end{thm}

\begin{proof}
Let $\mathbb Q^+$ and $\mathbb Q^-$ be two disjoint copies of $\mathbb Q$. Denote the elements corresponding to $q \in \mathbb Q$ in these copies by $q^+$ and $q^-$, respectively. Consider the (undirected) graph $G = (V,E)$ where $V =\mathbb Q^+ \cup \mathbb Q^-$ and $q^+r^- \in E$  whenever $q < r$. Figure \ref{fig:induced} shows a small subgraph of this graph to give an idea of what it looks like. Clearly $G$ is countable and connected.

Note that $G$ is bipartite with bipartition $\mathbb Q^+ \cup \mathbb Q^-$. Hence every automorphism $\gamma$ of $G$ either fixes $\mathbb Q^+$ and $\mathbb Q^-$ set-wise, or swaps the two sets. Furthermore if $\gamma q^+ = r^+$ then $\gamma q^- = r^-$ because $q^-$ is the unique vertex with the property $N(q^-) = \bigcap_{v \sim q^+} N(v) \setminus \{q^+\}$. A similar argument shows that if $\gamma q^+ = r^-$ then $\gamma q^- = r^+$. So the action on $\mathbb Q^+$ uniquely determines an automorphism of $G$.

Now, we define a family of automorphisms of $G$ (we will later show that these are in fact all the automorphisms of $G$). For every order automorphism $\gamma$ of $\mathbb Q$, define the functions $\gamma_\uparrow$ and $\gamma_\downarrow$ as follows: 
\begin{itemize}
\item $\gamma_\uparrow$ applies $\gamma$ to both copies of $\mathbb Q$, i.e.\ $\gamma_\uparrow(q^+) = (\gamma(q))^+$, $\gamma_\uparrow(q^-) = (\gamma(q))^-$,
\item $\gamma_\downarrow$ first applies $\gamma$ to both copies, then reverses the order on each of them and swaps them, i.e.\ $\gamma_\downarrow(q^+) = (-\gamma(q))^-$, and $\gamma_\downarrow(q^-) = (-\gamma(q))^+$.
\end{itemize}
It is straightforward to check that these maps are indeed automorphisms of the graph $G$.

To see that $G$ is arc transitive, notice that the arc $0^+1^-$ can be mapped to any arc of the form $q^+r^-$ by the automorphism $\gamma_\uparrow$ where
\[
	\gamma(x) = q+(r-q)x.
\]
The map $\gamma$ is an order automorphism of $\mathbb Q$ since $q^+r^- \in E$ implies that $q < r$. By analogous arguments, the arc $0^+1^-$ can be mapped to any arc of the form $q^-r^+$ by the automorphism $\gamma_\downarrow$ where
\[
	\gamma(x) = -q+(q-r)x.
\]

Every map of the type $\gamma_\uparrow$ and $\gamma_\downarrow$ moves infinitely many vertices. Thus, to show that $G$ has infinite motion, it suffices to prove that the automorphisms of the form $\gamma_\uparrow$ and $\gamma_\downarrow$ as defined above are the only automorphisms of $G$.

It is not hard to see that $q \geq r$ if and only if $N(q^+) \subseteq N(r^+)$. This implies that $N(\phi(q^+))\subseteq N(\phi(r^+))$ for every automorphism $\varphi$ of $G$. If $\varphi$ fixes $\mathbb Q^+$ set-wise we conclude that $\varphi$ preserves the order on $\mathbb Q^+$, hence it is equal to $\gamma_\uparrow$ for a suitable order automorphism $\gamma$. An analogous argument shows that if  $\varphi$ swaps $\mathbb Q^+$ and $\mathbb Q^-$, then $\varphi = \gamma_\downarrow$ for an order automorphism $\gamma$ of $\mathbb Q$.

Finally, assume that there is a distinguishing colouring $c$ of $G$ with $n < \infty$ colours. In particular this colouring would break every automorphism of the form $\gamma_\uparrow$. Hence the map $q \mapsto (c(q^+), c(q^-))$ would be a distinguishing colouring of $\mathbb Q$ with $n^2 < \infty$ colours, a contradiction to Theorem \ref{thm:dist-q}.
\end{proof}

\bibliographystyle{abbrv}
\bibliography{sources}

\end{document}